\documentclass[reqno]{amsart}
\usepackage{hyperref}

\AtBeginDocument{{\noindent\small
\emph{Electronic Journal of Differential Equations},
Vol. 2016 (2016), No. 335, pp. 1--9.\newline
ISSN: 1072-6691. URL: http://ejde.math.txstate.edu or http://ejde.math.unt.edu}
\thanks{\copyright 2016 Texas State University.}
\vspace{8mm}}

\begin{document}
\title[\hfilneg EJDE-2016/335\hfil Orbital stability of Gausson solutions]
{Orbital stability of Gausson solutions to logarithmic Schr\"odinger equations}

\author[A. H. Ardila \hfil EJDE-2016/335\hfilneg]
{Alex H. Ardila}

\address{Alex H. Ardila \newline
Department of Mathematics,
IME-USP, Cidade Universit\'aria,
CEP 05508-090, S\~ao Paulo, SP, Brazil}
\email{alexha@ime.usp.br}

\thanks{Submitted July 15, 2016. Published December 28, 2016.}
\keywords{Logarithmic Schr\"odinger equation; standing waves;
orbital stability; \hfill\break\indent variational method}

\begin{abstract}
 In this article we  prove of the  orbital stability of the ground state
 for logarithmic Schr\"odinger equation in any dimension and under
 nonradial perturbations. This  general stability result was announced
 by Cazenave and Lions \cite[Remark II.3]{CALO}, but no details were given
 there.
\end{abstract}

\maketitle
\numberwithin{equation}{section}
\newtheorem{theorem}{Theorem}[section]
\newtheorem{lemma}[theorem]{Lemma}
\newtheorem{proposition}[theorem]{Proposition}
\newtheorem{remark}[theorem]{Remark}
\newtheorem{definition}[theorem]{Definition}
\allowdisplaybreaks

\section{Introduction}

In this article we study the logarithmic Schr\"odinger equation
\begin{equation}\label{0NL}
  i\partial_{t}u+\Delta u+u\log |u|^2=0,
\end{equation}
where  $u=u(x,t)$ is a complex-valued function of 
$(x,t)\in \mathbb{R}^N\times\mathbb{R}$, $N\geq1$.
This equation  was proposed by Bialynicki-Birula and Mycielski \cite{CAS} 
in 1976 as a model of nonlinear wave mechanics. It laso has several applications
in quantum mechanics, quantum optics, nuclear physics, open quantum systems
 and  Bose-Einstein  condensation (see e.g. \cite{APLES} and the references therein).  
Recently,  \eqref{0NL} has proved useful for the modeling of several 
nonlinear phenomena including  geophysical applications of magma transport
 \cite{MFG} and nuclear physics \cite{HE}.

The mathematical literature concerning the logarithmic Schr\"odinger equation
does not seem to be very extensive. The Cauchy problem for \eqref{0NL} was 
tread by Cazenave and Haraux \cite{TA} in a suitable functional framework. 
Cazenave \cite{CL}; Cazenave and Lions \cite{CALO}; Blanchard and co.\ 
\cite{PHBJ, PHST}; research the  stability properties of standing waves 
for \eqref{0NL}.  In recent years, the logarithmic NLS equation  has attracted  
some attention  both in the theoretical and the applied mathematical literature.
 Among such works, let us mention
\cite{AHAA, PMS, CZ, SMSA, PDAMS}.

The energy functional $E$ associated with problem \eqref{0NL} is
\begin{equation}\label{EEEE}
E(u)=\frac{1}{2}\int_{\mathbb{R}^N}|\nabla u|^2dx
-\frac{1}{2}\int_{\mathbb{R}^N}|u|^2\log |u|^2dx.
\end{equation}
Unfortunately, because of the singularity of the logarithm at the origin, 
the functional fails to be finite as well of class $C^{1}$ on 
$H^{1}(\mathbb{R}^N)$. Because of  this  loss  of  smoothness,
it is convenient  to work in a suitable Banach space endowed with a 
Luxemburg type norm  to make functional $E$  well defined and $C^{1}$ smooth. 
This space allow to control the singularity of the logarithmic
nonlinearity at infinity and at the origin. Indeed, we consider the reflexive 
Banach space (see Appendix below)
\begin{equation}\label{ASE}
W(\mathbb{R}^N)=\big\{u\in H^{1}(\mathbb{R}^N):|u|^2\log |u|^2
\in L^{1}(\mathbb{R}^N)\big\},
\end{equation}
then it is well known that the energy functional $E$ is well-defined and of 
class $C^1$ on $W({\mathbb R}^N)$ (see \cite{CL}).
Moreover, Cazenave \cite[Theorem 9.3.4]{CB} proved the global well-posedness 
of the Cauchy problem for \eqref{0NL} in the energy space $W({\mathbb R}^N)$.

\begin{proposition} \label{PCS}
For each $u_0\in {W}({\mathbb{R}}^N)$, there is a unique maximal solution
$u$ of equation \eqref{0NL} such that 
$u\in C(\mathbb{R},{W}({\mathbb{R}^N}))\cap C^{1}(\mathbb{R}, {W}'
({\mathbb{R}^N}))$, $u(0)=u_0$ and 
$\sup_{t\in \mathbb{R}}\|u(t)\|_{{W}({\mathbb{R}^N})}<\infty$. 
Furthermore, the conservation of energy and charge hold; that is,
\begin{equation*}
E(u(t))=E(u_0)\quad  \text{and} \quad
\|u(t)\|^2_{L^2}=\|u_0\|^2_{L^2}\quad  \text{for all $t\in \mathbb{R}$}.
\end{equation*}
\end{proposition}

Let $\omega\in \mathbb{R}$ and $\varphi\in {W} ({{\mathbb{R}^N}})$
be solutions of the semilinear elliptic equation
\begin{equation}\label{EP}
-\Delta \varphi+\omega \varphi-\varphi\, \log |\varphi |^2=0, \quad
x\in {\mathbb{R}}^N,
\end{equation}
then, $u(x,t)=e^{i\omega t}\varphi(x)$ is a standing wave of \eqref{0NL}.  
It is well known (see \cite{CAS}) that the Gausson
\begin{equation}\label{IGS1}
\phi_{\omega}(x):=e^{\frac{\omega+N}{2}}e^{-\frac{1}{2}|x|^2},
\quad x\in {\mathbb{R}}^N,
\end{equation}
solves \eqref{EP} for any dimension $N$. Up to translations, 
\eqref{IGS1} is the unique strictly positive $C^2$-solution for \eqref{EP}
such that $\varphi(x)\to 0$ as $|x|\to \infty$.
Moreover, it is nondegenerate; that is, the dimension of the nullspace 
of the linearized operator is $N$, i.e. smallest possible (see \cite{PMS}).

The orbital stability of the Gausson \eqref{IGS1} when $N\geq2$ has been 
studied in \cite{PHBJ,PHST,CL}. In particular, Cazenave \cite{CL} proved 
that $e^{i\omega t}\phi_{\omega}(x)$ is stable in ${W} ({{\mathbb{R}^N}})$,
with respect to radial perturbations, for $N\geq2$. Their argument is based 
on the fact that the space of radially symmetric functions in 
${W}({{\mathbb{R}^N}})$ is compactly embedded into $L^2({{\mathbb{R}^N}})$ 
for $N\geq2$. Other proof, for $N\geq3$ and under radial perturbations, was 
given in \cite{PHBJ, PHST}. This proof  relies on application of the Shatah 
formalism \cite{SHH}.

As we have mentioned, Cazenave and Lions \cite[Remark II.3]{CALO}  claimed 
that the Gausson \eqref{IGS1} is orbitally stable in the unrestricted space 
${W} ({{\mathbb{R}^N}})$  for all $N\geq1$, but there the proof is omitted. 
The main aim of this paper is to give a detailed proof of this fact.

The notions of stability and instability are defined as follows.

\begin{definition}\label{2D111} \rm
We say that  a standing wave solution $u(x,t)=e^{i\omega t}\phi(x)$ of \eqref{0NL} 
is orbitally stable in ${W} ({{\mathbb{R}^N}})$ if for any  $\epsilon>0$ 
there exist $\eta>0$  such that if $u_0\in {W} ({{\mathbb{R}^N}})$ and
$\|u_0-\varphi \|_{{W} ({{\mathbb{R}^N}})}<\eta$, then the solution $u(t)$
of  \eqref{0NL}  with $u(0)=u_0$ exist for all $t\in \mathbb R$ and satisfies
\begin{equation*}
\sup_{t\in \mathbb{R}}\inf_{\theta\in \mathbb{R}}
\inf_{y\in \mathbb{R}^N}  \|u(t)-e^{i\theta}\varphi(\cdot- y) \|_{W(\mathbb{R}^N)}
<\epsilon.
\end{equation*}
Otherwise, the standing wave $e^{i\omega t}\phi(x)$ is said to be  unstable in 
${W} ({{\mathbb{R}^N}})$.
\end{definition}

 Before we state our result, we establish a variational characterization of the
 Gausson \eqref{IGS1}.
For $\omega\in \mathbb{R}$, we define the following functionals of class 
$C^{1}$ on $W(\mathbb{R}^N)$:
\begin{gather*}
 S_{\omega}(u)=\frac{1}{2}\int_{\mathbb{R}^N}|\nabla u|^2dx
 +\frac{\omega+1}{2}\int_{\mathbb{R}^N}|u|^2dx
 -\frac{1}{2}\int_{\mathbb{R}^N}|u|^2\log |u|^2dx,\\
  I_{\omega}(u)=\int_{\mathbb{R}^N}|\nabla u|^2dx
 +\omega\int_{\mathbb{R}^N}|u|^2dx-\int_{\mathbb{R}^N}|u|^2\log |u|^2dx.
\end{gather*}
Note that \eqref{EP} is equivalent to $S'_{\omega}(\varphi)=0$, and 
$I_{\omega}(u)=\langle S_{\omega}'(u),u\rangle$ is the so-called Nehari
functional.

Moreover, we consider the minimization problem
\begin{equation}
\begin{split} \label{MPE}
d(\omega)
&={\inf}\{S_{\omega}(u):\, u\in W(\mathbb{R}^N) \setminus  \{0 \},
  I_{\omega}(u)=0\} \\
&=\frac{1}{2}\,{\inf}\{\|u\|_{L^2}^2:u\in  W(\mathbb{R}^N) \setminus \{0 \},
  I_{\omega}(u)= 0 \},
\end{split}
\end{equation}
and define the set of ground states  by
\begin{equation*}
 \mathcal{N}_{\omega}=\bigl\{ \varphi\in W(\mathbb{R}^N) \setminus  \{0 \}
: S_{\omega}(\varphi)=d(\omega), \,\, I_{\omega}(\varphi)=0\bigl\}.
\end{equation*}
The set $\bigl\{u\in W(\mathbb{R}^N)\setminus  \{0 \},  I_{\omega}(u)=0\bigl\}$
is called the Nehari manifold. Notice that  the above set contains all
stationary point of $S_{\omega}$. In Section \ref{S:0}, we show that the
quantity $d(\omega)$ is positive for every $\omega\in \mathbb{R}$. Indeed,
\begin{equation*}
d(\omega)=\frac{1}{2}{\pi}^{N/2}e^{\omega+ N}.
\end{equation*}

\begin{remark}\label{RM} \rm
Let $u\in \mathcal{N}_{\omega}$.  Then, there exist a Lagrange multiplier  
$\Lambda\in \mathbb{R}$ such that $S'_{\omega}(u)=\Lambda I'_{\omega}(u)$. 
Thus, we have $\langle S'_{\omega}( u),u\rangle
=\Lambda\langle  I'_{\omega}(u),u\rangle$. The fact that   
$\langle S'_{\omega}( u),u\rangle =I_{\omega}(u)=0$ and 
$\langle  I'_{\omega}(u),u\rangle= -2\|u_{}\|_{L^2}^2<0$, implies 
$\Lambda=0$; that is, $S_{\omega}'(u)=0$. Therefore,  $u$ satisfies \eqref{EP}.
\end{remark}

The existence of minimizers for the minimization problem \eqref{MPE} is proved 
by the standard variational argument. We will show the following proposition 
in the Section \ref{S:0}.

\begin{proposition} \label{ESSW}
There exists a minimizer of $d(\omega)$ for any $\omega \in \mathbb{R}$. 
 Moreover, the set of ground states is  given by 
$\mathcal{N}_{\omega}=\{e^{i\theta}\phi_{\omega}(.-y); 
\theta\in\mathbb{R}, y\in\mathbb{R}^N \}$, where $\phi_{\omega}$ 
is  given in \eqref{IGS1}.
\end{proposition}
We remark that  Proposition \ref{ESSW} was claimed without proof by 
Cazenave \cite[Remark 9.3.8]{CB}.  It is also important to note  that 
the ground state be unique up to translations and phase shifts. 
In higher dimensions, it is known that there exist infinitely many weak 
solutions $u_n\in H^{1}(\mathbb{R}^N)$ of \eqref{EP} such that
$S_{\omega}(u_n)\to +\infty$ as $n\to +\infty$ (see e.g. \cite[Theorem 1.1]{PMS}).
The variational characterization of the Gausson \eqref{IGS1} as a minimizer of 
$S_{\omega}$ on the Nehari manifold,  contained in Proposition \ref{ESSW}, 
will be useful when we will deal with the stability.

Now we state our main result of this paper.

\begin{theorem} \label{2ESSW}
Let $\omega \in \mathbb{R}$ and $N\geq1$. Then  the standing wave 
$e^{i\omega t}\phi_{\omega}(x)$ is orbitally stable in
$W(\mathbb{R}^N)$.
\end{theorem}

The rest of the article is organized as follows. 
In Section \ref{S:0} we prove, by variational techniques, the existence 
of a minimizer for $d(\omega)$ (Proposition \ref{ESSW}). 
Section \ref{S:3} is devoted to the proof of Theorem \ref{2ESSW}. 
In the Appendix we include some information about of the space
 $W ({\mathbb{R}}^N)$.
\smallskip

\noindent{\bf Notation.} The space $L^2(\mathbb{R}^N,\mathbb{C})$  
will be denoted  by $L^2(\mathbb{R}^N)$ and its norm by $\|\cdot\|_{L^2}$.  
This space will be endowed  with the real scalar product
\begin{equation*}
 (u,v)=\Re\int_{\mathbb{R}^N}u\overline{v}\, dx, \quad 
\mathrm{for }\ u,v\in L^2(\mathbb{R}^N).
\end{equation*}
The space $H^{1}(\mathbb{R}^N,\mathbb{C})$  will be denoted  by 
$H^{1}(\mathbb{R}^N)$ and its norm by $\|\cdot\|_{H^{1}(\mathbb{R}^N)}$.
$\langle \cdot , \cdot \rangle$ is the duality pairing between $X'$ and $X$,
 where $X$ is a Banach space and $X'$ is its dual. Finally, 
$2^{\ast}:=2N/(N-2)$ if $N\geq3$ and $2^{\ast}:=+\infty$ if $N=1$ or $N=2$.
 Throughout this paper, the letter $C$ will denote positive constants.

\section{Existence and uniqueness of ground state}
\label{S:0}

Before giving the proof of Proposition \ref{ESSW}, some preparation is necessary. 
First, we recall the logarithmic Sobolev inequality. For a proof we refer to
 \cite[Theorem 8.14]{ELL}.

\begin{lemma} \label{L1}
Let $f$ be any function in $ H^{1}(\mathbb{R}^N)\setminus\{0\}$ and $\alpha$ 
be any positive number. Then,
\begin{equation}\label{SII}
\begin{aligned}
&\int_{\mathbb{R}^N}|f(x)|^2\log |f(x)|^2dx \\
&\leq \frac{\alpha^2}{\pi} \|\nabla f \|^2_{L^2}
 +(\log  \|f \|^2_{L^2}-N(1+\log \,\alpha))\|f \|^2_{L^2}.
\end{aligned}
\end{equation}
Moreover, there is equality if and only if $f$ is, up to translation, 
a multiple of  $e^{\{-\pi |x|^2/2 \alpha^2\}}$.
\end{lemma}

\begin{lemma}\label{L2}
Let $\omega\in \mathbb{R}$. Then, the quantity $d(\omega)$ is positive and satisfies
\begin{equation}\label{EA}
d(\omega)\geq \frac{1}{2}{\pi}^{N/2}e^{\omega+N}.
\end{equation}
\end{lemma}

\begin{proof}
 Let $u\in W(\mathbb{R}^N)  \setminus  \{0 \}$ be such that  $I_{\omega}(u)=0$. 
Using the logarithmic Sobolev inequality with $\alpha=\sqrt{{\pi}{}}$, 
we see that
\begin{equation*}
(\omega+N(1+\log (\sqrt{\pi})))\|u\|^2_{L^2}\leq (\log \|u\|^2_{L^2})\|u\|^2_{L^2},
\end{equation*}
which implies that  $\|u\|^2_{L^2}\geq {\pi}^{N/2}e^{\omega+N}$. Thus, by the 
definition of $d(\omega)$ given in \eqref{MPE}, we obtain \eqref{EA}.
\end{proof}

The following lemma is a variant of the Br\'ezis-Lieb lemma from \cite{LBL}.

\begin{lemma} \label{L4}
Let  $\{u_n\}$ be a bounded sequence in $W(\mathbb{R}^N)$ such that
$u_n\to u$ a.e. in $\mathbb{R}^N$. Then $u\in W(\mathbb{R}^N)$ and
\begin{equation*}
\lim_{n\to \infty}\int_{\mathbb{R}^N}\{|u_n|^2\log |u_n|^2-|u_n
 -u|^2\log |u_n-u|^2\}dx=\int_{\mathbb{R}^N}|u|^2\log |u|^2dx.
\end{equation*}
\end{lemma}

\begin{proof}
We first recall that, by \eqref{IFD} in the Appendix,  
$|z|^2\log |z|^2=A(|z|)-B(|z|)$  for every  $z\in\mathbb{C}$. 
We need only apply the  Br\'ezis-Lieb lemma (see Lemma \ref{lemmalieb}) 
to the functions $A$ and  $B$.  By the weak-lower semicontinuity of the 
${L^2(\mathbb{R}^N)}$-norm and Fatou lemma we have $u\in W(\mathbb{R}^N)$.  
It is clear that the sequence $\{u_n\}$ is  bounded in ${L^{A}(\mathbb{R}^N)}$.
 Since $A$ is convex and increasing function with $A(0)=0$, it is follows  
by property \eqref{DA1} in the Appendix that the function $A$ satisfies the 
assumptions of Lemma \ref{lemmalieb} in Appendix. Thus,
\begin{equation}\label{AC}
\lim_{n\to \infty}\int_{\mathbb{R}^N} | A(|u_n|)-A(|u_n-u|)- A(|u_{}|)|dx=0.
\end{equation}
On the other hand, by the continuous embedding 
$W(\mathbb{R}^N)\hookrightarrow H^{1}(\mathbb{R}^N)$, we have  that 
 $\{u_n\}$ is also bounded in ${H^{1}(\mathbb{R}^N)}$.
By H\"older an Sobolev inequalities, for any $u$, $v\in H^{1}(\mathbb{R}^N)$
we have that (see \cite[Lemma 1.1]{CL})
\begin{equation}\label{DB}
\int_{\mathbb{R}^N}|B(|u(x)|)- B(|v(x)|)|dx
\leq C(1+ \|u\|^2_{H^{1}(\mathbb{R}^N)}
+ \|v\|^2_{H^{1}(\mathbb{R}^N)} )\|u-v\|_{{L^2}}.
\end{equation}
Thus, the function $B$ satisfies the hypotheses (ii) and (iii) of 
Lemma \ref{lemmalieb}.  Furthermore, an easy calculation shows that the 
function $B$ is convex, increasing and  nonnegative with  $B(0)=0$. 
Then from Lemma \ref{lemmalieb} we see that
\begin{equation}\label{BC}
\lim_{n\to \infty}\int_{\mathbb{R}^N} | B(|u_n|)-B(|u_n-u|)- B(|u_{}|)|dx=0.
\end{equation}
Thus the result follows from \eqref{AC} and  \eqref{BC}.
\end{proof}

\begin{lemma} \label{L5}
Let $2<p<2^{\ast}$ and $\omega\in \mathbb{R}$. Assume that 
$\{ u_n\}\subset W(\mathbb{R}^N)$,  $I_{\omega}(u_n)=0$ for any
$n\in \mathbb{N}$ and $S_{\omega}(u_n)\to d(\omega)$ as $n$ approaches $+\infty$.
Then, there exist a constant $C>0$ depending only on $p$ such that 
$\|u_n\|^{p}_{L^{p}(\mathbb{R}^N)}\geq C$ for every $n\in \mathbb{N}$.
\end{lemma}

The proof of the above lemma follows along the same lines as \cite[Lemma 3.3]{CL}.
 We omit it.

\begin{proof}[Proof of Proposition \ref{ESSW}]
 Let $\{ u_n\} \subseteq W(\mathbb{R}^N)$ be a minimizing sequence for
$d(\omega)$, then the sequence $\{ u_n\}$ is bounded in  $W(\mathbb{R}^N)$.
Indeed, it is clear that the sequence $\|u_n\|^2_{L^2}$ is bounded.
Moreover, using the logarithmic Sobolev inequality and recalling that 
$I_{\omega}(u_n)=0$, we obtain
\begin{equation*}
\Big(1-\frac{\alpha^2}{\pi}\Big)\|\nabla u_n\|^2_{L^2}
\leq \Big(\log (\frac{e^{-(\omega+N)}}{\alpha^N})\Big)\|u_n\|^2_{L^2} 
+\big(\log \|u_n\|^2_{L^2} \big)\|u_n\|^2_{L^2}.
\end{equation*}
Taking $\alpha>0$ sufficiently small, we see that $\|\nabla u_n\|^2_{L^2}$
 is bounded, so the sequence $\{ u_n\}$ is bounded in $H^{1}(\mathbb{R}^N)$.  
Then, using  $I_{\omega}(u_n)=0$ again, and \eqref{DB} we obtain that there 
exist a constant $C>0$ such that
\begin{equation*}
\int_{\mathbb{R}^N}A(|u_n|)dx\leq \int_{\mathbb{R}^N}B(|u_n|)dx
+|\omega|\|u_n\|^2_{L^2}\leq C,
\end{equation*}
which implies, by \eqref{DA1} in the Appendix, that the sequence $\{ u_n\}$ 
is bounded in $W (\mathbb{R}^N)$.

Next, notice  that for any sequence $x_n\in \mathbb{R}^N$ we have that 
$\{ u_n(\cdot + x_n)\}$ is still a  bounded  minimizing sequence for $d(\omega)$.
 Moreover, if
\begin{equation*}
 \lim_{n \to +\infty}\sup_{y\in\mathbb{R}^N} \int _{B_{1}(y)}|u_n|^2dx =0,
\end{equation*}
then $u_n\to 0$ in $L^{p}(\mathbb{R}^N)$ for any $p\in (2, 2^{\ast})$, where 
$B_{1}(y)=\{z\in\mathbb{R}^N: |y-z|<1 \}$. Therefore, from Lemma \ref{L5} 
and the compactness of the embedding $H^{1}(B_{1}(0))\hookrightarrow L^2(B_{1}(0))$,
we deduce that there exist a sequence $y_n\in \mathbb{R}^N$ such that the weak 
limit  in $H^{1}(\mathbb{R}^N)$ of the sequence $\{ u_n(\cdot + y_n)\}$ 
is not the trivial function. Let $v_n:=u_n(\cdot+ y_n)$. 
Then  there exist $\varphi \in W(\mathbb{R}^N)\setminus\{0\}$ such that, 
up to a subsequence, $v_n\rightharpoonup \varphi$ weakly in $W (\mathbb{R}^N)$ 
and $v_n\to \varphi$ a.e.\ in  $\mathbb{R}^N$.

Now we prove that $I_{\omega}(\varphi)=0$ and $S(\varphi)=d(\omega)$. 
First,  assume by contradiction that $I_{\omega}(\varphi)<0$. 
By elementary computations, we can see that there is $0<\lambda<1$ 
such that $I_{\omega}(\lambda \varphi)=0$. Then, from the definition of 
$d(\omega)$ and  the weak lower semicontinuity of the $L^2(\mathbb{R}^N)$-norm, 
we have
\begin{equation*}
d(\omega)\leq \frac{1}{2}\|\lambda \varphi\|^2_{L^2}
<\frac{1}{2}\|\varphi\|^2_{L^2}
\leq \frac{1}{2}\liminf_{n\to \infty}\|v_n\|^2_{L^2}=d(\omega),
\end{equation*}
it which is impossible. On the other hand, assume that $I_{\omega}(\varphi)>0$. 
Since the embedding  $W(\mathbb{R}^N)\hookrightarrow {H^{1}(\mathbb{R}^N)}$ 
is continuous, we see that $v_n\rightharpoonup \varphi$ weakly in 
$H^{1}(\mathbb{R}^N)$. Thus, we have
\begin{gather}
 \|v_n\|^2_{L^2}-\|v_n-\varphi\|^2_{L^2}-\|\varphi\|^2_{L^2}\to0,  \label{2C11}\\
\|\nabla v_n\|^2_{L^2}-\|\nabla v_n-\nabla\varphi\|^2_{L^2}
-\|\nabla \varphi \|^2_{L^2}\to0\label{2C12}
\end{gather}
as $n\to\infty$. Combining \eqref{2C11}, \eqref{2C12} and Lemma \ref{L4} leads to
\begin{equation*}
\lim_{n\to \infty}I_{\omega}(v_n-\varphi)
=\lim_{n\to \infty}I_{\omega}(v_n)-I_{\omega,\gamma}(\varphi)
=-I_{\omega}(\varphi),
\end{equation*}
which combined with  $I_{\omega}(\varphi)> 0$ give us  that 
$I_{\omega}(v_n-\varphi)<0$ for sufficiently large $n$. Thus, by \eqref{2C11} 
and  applying the same argument as above, we see that
\begin{equation*}
d(\omega)\leq\frac{1}{2} \lim_{n\to \infty}\|v_n-\varphi\|^2_{L^2}
=d(\omega)-\frac{1}{2}\|\varphi\|^2_{L^2},
\end{equation*}
which is a contradiction because $\|\varphi\|^2_{L^2}>0$. Then, we deduce 
that $I_{\omega}(\varphi)=0$. In addition, by the weak lower semicontinuity 
of the $L^2(\mathbb{R}^N)$-norm,  we have
\begin{equation}\label{inequa}
d(\omega)\leq \frac{1}{2}\|\varphi\|^2_{L^2}
\leq \frac{1}{2} \liminf_{n\to \infty}\|v_n\|^2_{L^2}=d(\omega),
\end{equation}
which implies, by the definition of $d(\omega)$, that 
$\varphi\in \mathcal{N}_{\omega}$. This proves the first part of the 
statement of Proposition \ref{ESSW}.

By direct computations, we see that the Gausson \eqref{IGS1} 
 satisfies $I_{\omega}(\phi_{\omega})=0$ and 
$S_{\omega}(\phi_{\omega})= {\pi}^{N/2}e^{\omega+N}/2$. 
Thus, by Lemma \ref{L2} we have that $d(\omega)= {\pi}^{N/2}e^{\omega+N}/2$ 
and \begin{equation*}\{e^{i\theta}\phi_{\omega}(\cdot-y); 
\theta\in\mathbb{R},\, y\in\mathbb{R}^N \}\subseteq\mathcal{N}_{\omega}.
\end{equation*}
 Next, let $\varphi\in \mathcal{N}_{\omega}$. Then, by definition of 
$d(\omega)$, $\|\varphi\|^2_{L^2}=2d(\omega)= {\pi}^{N/2}e^{\omega+N}$ 
and  $I_{\omega}(\varphi)=0$.  This implies that $\varphi$ satisfies 
the equality  in \eqref{SII} with $\alpha=\sqrt{\pi}$. 
Indeed, suppose that we have the strict inequality in \eqref{SII}
 with $\alpha=\sqrt{\pi}$. Since $\varphi$ satisfies $I_{\omega}(\varphi)=0$,
 it is easy to show that in this case  $\|\varphi\|^2_{L^2}>{\pi}^{N/2}e^{\omega+N}$, 
it which is impossible. Therefore, from Lemma \ref{L1} we infer that there 
exist $r>0$, $\theta_0\in \mathbb{R} $ and $y\in \mathbb{R}^N$ such that
 \begin{equation*}
\varphi(x)=r\, e^{i\theta_0}e^{-\frac{1}{2} |x-y_{}|^2}.
\end{equation*}
Elementary calculations show that $r^2=e^{\omega +N}$. Thus, we have 
$\varphi(x)= e^{i\theta_0}\,\phi_{\omega}(x-y_{})$ and
 Proposition \ref{ESSW} is proved.
\end{proof}

\section{Stability of ground state}\label{S:3}

The proof of Theorem \ref{2ESSW} relies on the following compactness result.

\begin{lemma} \label{CSM}
Let $\{ u_n\}\subseteq W (\mathbb{R}^N)$ be a minimizing sequence for $d(\omega)$. 
Then there exist a family $(y_n)\subset \mathbb{R}^N$ and a function  
$\varphi\in\mathcal{N}_{\omega}$ such that, possibly for a subsequence only,
\begin{equation*}
u_n(\cdot-y_n)\to \varphi \quad \text{strongly in $W(\mathbb{R}^N)$}.
\end{equation*}
\end{lemma}

\begin{proof}
By Proposition \ref{ESSW}, we see that exist  $(y_n)\subset \mathbb{R}^N$ 
and a function  $\varphi\in\mathcal{N}_{\omega}$ such that, up to a subsequence, 
$u_n(\cdot-y_n)\rightharpoonup \varphi$ weakly  in $W({\mathbb{R}^N})$. 
Let $v_n:=u_n(\cdot-y_n)$. From \eqref{2C11},
we infer that  $v_n\to \varphi$  in $L^2(\mathbb{R}^N)$. Then, since the 
sequence $\{ v_n\}$ is bounded in $H^{1}(\mathbb{R}^N)$, from \eqref{DB} we obtain
\begin{equation*}
 \lim_{n\to \infty}\int_{\mathbb{R}^N}B(|v_n(x)|)dx
=\int_{\mathbb{R}^N}B(|\varphi(x)|)dx,
\end{equation*}
which combined with $I_{\omega}(v_n)=I_{\omega}(\varphi)=0$  for any 
$n\in \mathbb{N}$,  gives
\begin{equation}\label{2BX1}
\lim_{n\to \infty}\Big[ \int_{\mathbb{R}^N}|\nabla v_n|^2dx
+\int_{\mathbb{R}^N}A(|v_n(x)|)dx\Big]
= \int_{\mathbb{R}^N}|\nabla \varphi|^2dx+\int_{\mathbb{R}^N}A(|\varphi(x)|)dx.
\end{equation}
Moreover, by \eqref{2BX1},  the weak lower semicontinuity of the 
$L^2(\mathbb{R}^N)$-norm and Fatou lemma, we deduce 
(see e.g. \cite[Lemma 12 in chapter V]{AH})
\begin{gather}
 \lim_{n\to \infty}\int_{\mathbb{R}^N}|\nabla v_n|^2dx
=\int_{\mathbb{R}^N}|\nabla \varphi|^2dx,\label{N1}\\
 \lim_{n \to \infty}\int_{\mathbb{R}^N}A(|v_n(x)|)dx
=\int_{\mathbb{R}^N}A(|\varphi(x)|)dx. \label{N2} 
\end{gather}
Since $v_n\rightharpoonup \varphi$ weakly  in $H^{1}(\mathbb{R}^N)$, 
it follows from  \eqref{N1}  that $v_n\to\varphi$  in $H^{1}(\mathbb{R}^N)$.  
Finally, by Proposition  \ref{orlicz}-{ii)} in Appendix  and \eqref{N2} we have 
$v_n\to\varphi$  in $L^{A}(\mathbb{R}^N)$. Thus, by definition of the 
$W({\mathbb{R}^N})$-norm, we infer that $v_n\to\varphi$  in 
$W({\mathbb{R}}^N)$. Which completes the proof.
\end{proof}

\begin{proof}[Proof of Theorem \ref{2ESSW}]
The result is proved  by contradiction. Assume that there exist  $\epsilon>0$ 
and two  sequences $\{u_{n,0}\}\subset W(\mathbb{R}^N)$, 
$\{t_n\}\subset \mathbb{R}$ such that
\begin{gather}
\|u_{n,0}-\phi_{\omega}\|_{ W(\mathbb{R}^N)}\to 0, \quad \text{ as  }
  n\to \infty, \label{T21} \\
\inf_{\theta\in \mathbb{R}}\inf_{y\in \mathbb{R}^N}  
\|u(t_n)-e^{i\theta}\phi_{\omega}(\cdot- y) \|_{W(\mathbb{R}^N)}\geq \epsilon , 
\quad \text{for any $n\in \mathbb{N}$,}\label{T22}
\end{gather}
where $u_n$ is the solution of \eqref{0NL} with initial data $u_{n,0}$. 
Set $v_n(x)= u_n(x,t_n)$. By \eqref{T21} and conservation laws, we obtain
\begin{gather}
\|v_n\|^2_{L^2}=\|u_n(t_n)\|^2_{L^2}=\|u_{n,0}\|^2_{L^2}
\to \|\phi_{\omega}\|^2_{L^2}\label{CE1} \\
E(v_n)=E(u_n(t_n))=E(u_{n,0})\to E(\phi_{\omega}),
\label{CE2}
\end{gather}
as $n\to \infty$.  In particular, it follows from \eqref{CE1} and \eqref{CE2} that,
 as $n\to \infty$,
\begin{equation}\label{A12}
S_{\omega}(v_n)\to S_{\omega}(\phi_{\omega})=d(\omega).
\end{equation}
Moreover, by combining \eqref{CE1} and \eqref{A12} lead us to 
$I_{\omega}(v_n)\to 0$ as $n\to \infty$. Next,
define the sequence $f_n(x)=\rho_nv_n(x)$ with
\begin{equation*}
\rho_n=\exp(\frac{I_{\omega}(v_n)}{2\|v_n\|^2_{L^2}}),
\end{equation*}
where $\exp(x)$ represent the exponential function. 
It is clear that $\lim_{n\to \infty}\rho_n=1$ and $I_{\omega}(f_n)=0$ 
for any $n\in\mathbb{N}$. Furthermore, since the sequence $\{v_n\}$  
is bounded in $W(\mathbb{R}^N)$, we obtain $\|v_n-f_n\|_{W(\mathbb{R}^N)}\to 0$ 
as $n\to \infty$. Then, by  \eqref{A12}, we have that $\{f_n\}$ is a minimizing 
sequence for $d(\omega)$. Thus, by Lemma \ref{CSM}, up to a subsequence, 
there exist $(y_n)\subset \mathbb{R}^N$ and a function  
$\varphi\in\mathcal{N}_{\omega}$ such that
\begin{equation}\label{UEP1}
\| f_n(\cdot-y_n)- \varphi\|_{W(\mathbb{R}^N)}\to 0 \quad \text{as $n\to +\infty$}.
\end{equation}
Now, by Proposition \ref{ESSW}, there exist  $\theta_0\in \mathbb{R} $ and
$y_0\in \mathbb{R}^N$ such that
$\varphi(x)= e^{i\theta_0}\,\phi_{\omega}(x-y_0)$. Remembering that
$v_n= u_n(t_n)$ and using \eqref{UEP1}, we obtain
\begin{equation*}
\|u_n(t_n)-e^{i\theta_0}\phi_{\omega}(\cdot- (y_0-y_n))\|_{ W(\mathbb{R}^N)}
\to 0\quad \text{as $n\to +\infty$}
\end{equation*}
which contradicts \eqref{T22}. This completes the proof.
\end{proof}

\section{Appendix}\label{S:5}

The purpose of this Appendix is to describe  the structure of space 
${W}({\mathbb{R}}^N)$. We need to introduce some notation. Define
\begin{equation*}
F(z)=|z|^2\log |z|^2\quad  \text{for every  }z\in\mathbb{C},
\end{equation*}
and as in \cite{CL},  we define the functions  $A$, $B$ on $[0, \infty)$  by
\begin{equation}\label{IFD}
A(s)=
\begin{cases}
-s^2\,\log (s^2), &\text{if $0\leq s\leq e^{-3}$;}\\
3s^2+4e^{-3}s -e^{-6}, &\text{if $ s\geq e^{-3}$;}
\end{cases}
\qquad  B(s)=F(s)+A(s).
\end{equation}
 Note that $A$ is a nonnegative  convex and increasing function, and
 $A\in C^{1}([0,+\infty))\cap C^2((0,+\infty))$. The Orlicz space 
$L^{A}(\mathbb{R}^N)$ corresponding to $A$ is defined by
\begin{equation*}
L^{A}(\mathbb{R}^N)=\big\{u\in L^{1}_{\rm loc}(\mathbb{R}^N) 
: A(|u|)\in L^{1}_{}(\mathbb{R}^N)\big\},
\end{equation*}
equipped with the Luxemburg norm
\begin{equation*}
\|u\|_{L^{A}}=\inf\big\{k>0: \int_{\mathbb{R}^N}A(k^{-1}{|u(x)|})dx\leq 1 \big\}.
\end{equation*}
Here as usual $L^{1}_{\rm loc}(\mathbb{R}^N)$ is the space of all locally 
Lebesgue integrable functions. It is proved in \cite[Lemma 2.1]{CL} that $A$ 
is a Young-function which is $\Delta_{2}$-regular and 
$(L^{A}(\mathbb{R}^N),\|\cdot\|_{L^{A}} )$ is a separable reflexive  Banach space.

Next, we consider the reflexive Banach space 
$W({\mathbb{R}}^N)=H^{1}(\mathbb{R}^N)\cap L^{A}(\mathbb{R}^N)$ 
equipped with the usual norm 
$\|u\|_{W(\mathbb{R}^N)}=\|u\|_{H^{1}(\mathbb{R}^N)}+\|u\|_{L^{A}}$ (see \eqref{ASE}).
 It is easy to see that (see \cite[Proposition 2.2]{CL} for more details)
\begin{equation*}
W(\mathbb{R}^N)=\{u\in H^{1}(\mathbb{R}^N):|u|^2\log |u|^2\in L^{1}(\mathbb{R}^N)\}.
\end{equation*}
Furthermore, it is  known that  the dual space (see \cite[Proposition 1.1.3]{CB})
\begin{equation*}
{W}'({\mathbb{R}^N})=H^{-1}(\mathbb{R}^N)+L^{A'}(\mathbb{R}^N),
\end{equation*}
where the Banach space ${W}'({\mathbb{R}^N})$ is equipped with its usual norm.
 Here, $L^{A'}(\mathbb{R}^N)$ is the dual space of $L^{A}(\mathbb{R}^N)$
 (see \cite{CL}).

Now we list some properties of the Orlicz space $L^{A}(\mathbb{R}^N)$ 
that we have used above.  For a proof of such statements we refer to 
\cite[Lemma 2.1]{CL}.

\begin{proposition} \label{orlicz}
Let $\{u_{{m}}\}$ be a sequence in  $L^{A}(\mathbb{R}^N)$, the following facts hold:
\begin{itemize}
\item[(i)] If  $u_{{m}}\to u$ in $L^{A}(\mathbb{R}^N)$, then 
 $A(|u_{{m}}|)\to A(|u|)$ in $L^{1}(\mathbb{R}^N)$ as   $n\to \infty$.

\item[(ii)] Let  $u\in L^{A}(\mathbb{R}^N)$. If  $u_{m}\to u$ a.e. in
 $\mathbb{R}^N$ and if
\begin{equation*}
\lim_{n \to \infty}\int_{\mathbb{R}^N}A(|u_{m}(x)|)dx
=\int_{\mathbb{R}^N}A(|u(x)|)dx,
\end{equation*}
then $u_{{m}}\to u$ in $L^{A}(\mathbb{R}^N)$ as   $n\to \infty$.

\item[(iii)] For any $u\in L^{A}(\mathbb{R}^N)$, we have
\begin{equation}\label{DA1}
\min \{\|u\|_{L^{A}},\|u\|^2_{L^{A}}\}
\leq  \int_{\mathbb{R}^N}A(|u(x)|)dx\
leq \max\{\|u\|_{L^{A}},\|u\|^2_{L^{A}}\}.
\end{equation}
\end{itemize}
\end{proposition}

We conclude this Appendix with Br\'ezis-Lieb's lemma: see 
\cite[Theorem 2 and Example (b)]{LBL}

\begin{lemma} \label{lemmalieb}
Suppose that  $j$ is a continuous, convex function from $\mathbb{C}$ to 
$\mathbb{R}$ with $j(0)=0$ and let $f_n=f+g_n$ be a sequence of measurable 
functions from $\mathbb{R}^N$ to $\mathbb{C}$ such that: 
\begin{itemize}
\item[(i)] $g_n\to 0$ a.e. in $\mathbb{R}^N$.
\item[(ii)] $j(Mf)$ is in $L^{1}(\mathbb{R}^N)$ for every real $M$.
\item[(iii)] There exists some fixed $k>1$ such that $\{j(kg_n)-kj(g_n)\}$ 
is uniformly bounded in $L^{1}(\mathbb{R}^N)$.
\end{itemize}
Then
\begin{equation*}
\lim_{n\to \infty}\int_{\mathbb{R}^N} | j(f+g_n)-j(g_n)- j(f)|dx=0.
\end{equation*}
\end{lemma}

\subsection*{Acknowledgements}
The author is grateful  to the anonymous referee for pointing out several 
imprecisions. The author was supported by CAPES and CNPq/Brazil.

\end{document}